\documentclass[11pt]{amsart}
\usepackage{amssymb,amsthm}
\usepackage{mathtools}
\usepackage{eucal,stmaryrd}
\usepackage{enumerate}
\usepackage[margin=1in]{geometry}
\usepackage{quiver}
\usepackage{tabu,multirow}

\theoremstyle{plain}
\newtheorem{theorem}{Theorem}[section]
\newtheorem{proposition}[theorem]{Proposition}
\newtheorem{lemma}[theorem]{Lemma}

\theoremstyle{definition}
\newtheorem{definition}[theorem]{Definition}

\usepackage{hyperref}
\usepackage{xcolor}
\hypersetup{
    colorlinks,
    linkcolor={red!50!black},
    citecolor={blue!50!black},
    urlcolor={blue!80!black}
}

\newcommand{\tp}{{\scriptscriptstyle\mathsf{T}}}
\newcommand{\W}{{\scriptscriptstyle\mathsf{W}}}
\newcommand{\N}{{\scriptscriptstyle\mathsf{N}}}
\newcommand{\MP}{{\scriptscriptstyle\mathsf{MP}}}

\let\oh=\circ
\newcommand{\ccirc}{\mathbin{\mathchoice
  {\xcirc\scriptstyle}
  {\xcirc\scriptstyle}
  {\xcirc\scriptscriptstyle}
  {\xcirc\scriptscriptstyle}
}}
\newcommand{\xcirc}[1]{\vcenter{\hbox{$#1\oh$}}}
\let\circ\ccirc

\let\O\undefined
\let\S\undefined
\DeclareMathOperator{\O}{O}
\DeclareMathOperator{\V}{V}
\DeclareMathOperator{\rank}{rank}
\DeclareMathOperator{\S}{S}
\DeclareMathOperator{\SO}{SO}
\DeclareMathOperator{\Alt}{\mathsf{\Lambda}}
\DeclareMathOperator{\Sym}{\mathsf{S}}
\DeclareMathOperator{\tr}{tr}
\DeclareMathOperator{\Gr}{Gr}

\DeclareMathOperator{\diag}{diag}
\DeclareMathOperator{\Stab}{Stab}
\DeclareMathOperator{\Flag}{Flag}

\begin{document}
\title{Minimal equivariant embeddings of the Grassmannian and flag manifold}
\author[L.-H.~Lim]{Lek-Heng~Lim}
\address{Computational and Applied Mathematics Initiative, Department of Statistics,
University of Chicago, Chicago, IL 60637-1514.}
\email{lekheng@uchicago.edu}
\author[K.~Ye]{Ke Ye}
\address{KLMM, Academy of Mathematics and Systems Science, Chinese Academy of Sciences, Beijing 100190, China}
\email{keyk@amss.ac.cn}

\begin{abstract}
We show that the flag manifold $\Flag(k_1,\dots,  k_p, \mathbb{R}^n)$, with Grassmannian the special case $p=1$, has an $\SO_n(\mathbb{R})$-equivariant embedding in an Euclidean space of dimension $(n-1)(n+2)/2$, two orders of magnitude below the current best known result. We will show that the value $(n-1)(n+2)/2$ is the smallest possible and that any  $\SO_n(\mathbb{R})$-equivariant embedding of $\Flag(k_1,\dots,  k_p, \mathbb{R}^n)$ in an ambient space of minimal dimension is equivariantly equivalent to the aforementioned one.
\end{abstract}

\maketitle

\section{Introduction}\label{sec:intro}

Let $\Flag(k_1,\dots,k_p, \mathbb{R}^n)$ be the manifold of $(k_1,\dots,k_p)$-flags in $\mathbb{R}^n$. We will show that it may be embedded into $\Sym^2(\mathbb{R}^n)$, the space of $n \times n$ real symmetric matrices, as
\begin{equation}\label{eq:iso}
\Flag_{a_1,\dots,  a_{p+1}}(k_1,\dots,  k_p, n) \coloneqq
\left\lbrace
Q {\setlength{\arraycolsep}{0pt}
\begin{bmatrix}
a_1 I_{n_1} & 0 & \cdots & 0\\[-1ex]
0 &  a_2 I_{n_2} &  &  \vdots \\[-1ex]
\vdots &  & \ddots &  0 \\
0 & \cdots & 0 &  a_{p+1} I_{n_{p+1}}
\end{bmatrix}}
Q^\tp\in \Sym^2(\mathbb{R}^n) : Q\in \SO_n(\mathbb{R}) 
\right\rbrace, 
\end{equation}
where $a_1,\dots,  a_{p+1} \in \mathbb{R}$ are any fixed distinct constants. We call this the \emph{isospectral model} for flag manifolds, introduced in \cite{LK24a} for purely computational reasons, in large part to provide a more efficient platform for optimization algorithms \cite{KKL22}. To our surprise, we found that it also solves a problem in equivariant differential geometry.

The current best-known bound for embedding $\Flag(k_1,\dots,k_p, \mathbb{R}^n)$ isometrically into Euclidean space,  due to G\"{u}nther \cite{Gunther89a,Gunther89b,  91Gunther}, requires an ambient space of dimension at least
\begin{equation}\label{eq:m}
\max \Bigl\{ \frac12 m(m+3)+ 5,  \frac12 m(m+5) \Bigr\}, \quad m =\frac{1}{2} n^2 - \frac{1}{2}  \sum_{i=1}^{p+1} (k_i - k_{i-1})^2,
\end{equation}
where $m$ here is the dimension of $\Flag(k_1,\dots,  k_p, \mathbb{R}^n)$ and $0 \eqqcolon k_0 < k_1 < \dots < k_p < k_{p+1} \coloneqq n$ are integers.
Since we may choose the values of $a_1,\dots,  a_{p+1}$ in \eqref{eq:iso} so that $\tr(X) = 0$, we have an embedding  of $\Flag(k_1,\dots,k_p, \mathbb{R}^n)$  into  $\Sym^2_\oh(\mathbb{R}^n)$, the space of $n \times n$ real \emph{traceless} symmetric matrices, which has dimension
\begin{equation}\label{eq:our}
\frac12 (n-1)(n+2);
\end{equation}
and this embedding is both isometric and $\SO_n(\mathbb{R})$-equivariant.
Unlike G\"{u}nther's result, ours is constructive --- in fact the image of the embedding is explicitly given by \eqref{eq:iso}. It is straightforward to verify that the value in \eqref{eq:our} is strictly smaller than G\"{u}nther's bound \eqref{eq:m} for all possible values of $k_1,\dots, k_p$. We will prove that it is also the best possible $\SO_n(\mathbb{R})$-equivariant embedding of the flag manifold --- no such embedding can have ambient dimension lower than \eqref{eq:our}.

As a quick reminder of the various embedding problems in differential geometry: Given a manifold $\mathcal{M}$ of dimension $m$, one would like to find out if there is an embedding $\varepsilon : \mathcal{M} \to \mathbb{R}^d$ into $d$-dimensional Euclidean space with increasingly stronger requirements on $\varepsilon$. There are many celebrated results and we state a few of the most representative ones:
\begin{align*}
\tag*{Whitney \cite{Whitney44}:} \mathcal{M}&\text{ smooth}& \varepsilon &\text{ smooth} &d &= 2m;\\
\tag*{Nash \cite{Nash56}:} \mathcal{M}&\text{ Riemannian}& \varepsilon &\text{ isometric} &d &< \infty;\\
\tag*{Gromov \cite{Gromov86}:} \mathcal{M}&\text{ Riemannian}& \varepsilon &\text{ isometric} & d &= \tfrac12 m(m+5) + 3;\\
\tag*{G\"{u}nther \cite{Gunther89a, Gunther89b, 91Gunther}:} \mathcal{M}&\text{ Riemannian}& \varepsilon &\text{ isometric} & d &= \tfrac12 \max\{m(m+3) + 10,  m(m+5) \};\\
\tag*{Mostow--Palais \cite{Mostow57,Palais57}:} \mathcal{M}&\text{ $G$-manifold}& \varepsilon &\text{ equivariant} & d &<\infty.
\end{align*}
The last result requires $G$ to be a compact Lie group and is evidently not effective. The only explicit estimate we are aware of is a recent one of Wang \cite{Wang22}, which shows that for a finite group acting on a closed manifold, any embedding of $\mathcal{M}$ into $\mathbb{R}^d$ yields a $G$-equivariant embedding into $\mathbb{R}^{d |G|}$. When combined with Whitney's embedding, one gets a $2m |G|$-dimensional $G$-equivariant embedding.

For the specific case $\mathcal{M} = \Flag(k_1,\dots,  k_p, \mathbb{R}^n)$, the only related result we know is \cite{Lam75} but it gives only the existence of an \emph{immersion} into $\mathbb{R}^{\frac12 n(n-1)}$ for a manifold of flags of length $p \le n-2$, and into $\mathbb{R}^{\frac12 n(n-1) + 1}$ for the manifold of complete flags  $\Flag(1,2, \dots,  n-1, \mathbb{R}^n)$. For the special case $p = 1$, i.e., the Grassmannian $\Gr(k,\mathbb{R}^n)$, one may deduce lower bounds for $d$ from results about nonexistence of embeddings like \cite[Theorem~4.8]{MS74},   \cite[Propositions~4.1 and 4.2]{HS81},  \cite[Proposition~5.1]{Mayer99},  \cite[Theorems~1 and 2]{Oproiu76}. We will not survey these given that this article only concerns upper bounds for $d$.

Any $\SO_n(\mathbb{R})$-invariant Riemannian metric $\mathsf{g}$ on the flag manifold comes from its homogeneous space structure
\begin{equation}\label{eq:diffeo}
\Flag(k_1,\dots,  k_p, \mathbb{R}^n) \cong \SO_n(\mathbb{R})/\S(\O_{k_1}(\mathbb{R}) \times \O_{k_2 - k_1}(\mathbb{R}) \times \dots \times \O_{n - k_p}(\mathbb{R}))
\end{equation}
and is in one-to-one correspondence with an $\S(\O_{k_1}(\mathbb{R}) \times \O_{k_2 - k_1}(\mathbb{R}) \times \dots \times \O_{n - k_p}(\mathbb{R}))$-invariant inner product on 
\[
\mathfrak{m} \coloneqq \bigl\lbrace
(B_{ij}) \in \mathbb{R}^{n\times n}:  B_{ij} = -B_{ji}^\tp\in \mathbb{R}^{ (k_i-k_{i-1}) \times (k_j - k_{j-1})},\; B_{ii} = 0,\;1 \le i < j \le p+1
\bigr\rbrace,
\]
noting that $\mathfrak{so}(n) = [\mathfrak{so}(k_1) \oplus \mathfrak{so}(k_2 - k_1) \dots \oplus \mathfrak{so}(n - k_p)] \oplus \mathfrak{m}$. For any arbitrary distinct constants $a_1,\dots ,  a_{p+1}\in \mathbb{R}$,  we have an invariant inner product on $\mathfrak{m}$,
\[
\langle B,  C \rangle \coloneqq 2 \sum_{1 \le i < j \le p+1} (a_i-a_j)^2 \tr(B_{ij}^\tp C_{ij}),
\]
that uniquely determines a metric $\mathsf{g}$ on $\Flag(k_1,\dots,  k_p,  \mathbb{R}^n)$. In case one is wondering about the constants $a_1,\dots,  a_{p+1}$ appearing in \eqref{eq:iso}, this explains their origin. We will assume throughout this article that  $\Flag(n_1,\dots,  n_p, \mathbb{R}^n)$ is equipped with this $\SO_n(\mathbb{R})$-invariant  metric $\mathsf{g}$.

We will denote the dimensions of a minimal dimensional embedding according to Whitney, Nash, Mostow--Palais by 
\begin{align*}
d_\W &\coloneqq \min \{ d: \Flag(n_1,\dots,  n_p, \mathbb{R}^n)\text{ can be smoothly embedded in }\mathbb{R}^d \},  \\
d_\N &\coloneqq \min \{ d: \Flag(n_1,\dots,  n_p, \mathbb{R}^n)\text{ can be isometrically embedded in }\mathbb{R}^d \},  \\
d_\MP &\coloneqq \min \{ d: \Flag(n_1,\dots,  n_p, \mathbb{R}^n)\text{ can be $\SO_n(\mathbb{R})$-equivariantly embedded in }\mathbb{R}^d \},
\end{align*}
respectively.  It follows from our Theorem~\ref{thm:minimal flag} that
\begin{equation}\label{eq:boundN0} 
d_\W \le d_\N \le \frac12 (n-1)(n+2) = d_\MP
\end{equation}
and that \eqref{eq:iso} chosen to have $ a_1(k_1 -k_0) + \dots + a_{p+1} (k_{p+1} -k_p)  = 0$ is the \emph{unique} embedding that attains this. 
This improves the best known bound of Whitney embedding in some cases and that of Nash embedding in all cases; for Mostow--Palais embedding, our value gives $d_\MP$  exactly. To elaborate, our bound is better than Whitney's bound in \cite{Whitney44}, i.e.,
\[
\frac12 (n-1)(n+2) \le n^2 - \sum_{i=1}^{p+1} (k_i - k_{i-1})^2,
\] 
if and only if 
\[
 \sum_{i=1}^{p+1} (k_i - k_{i-1}) (k_i - k_{i-1} + 1)   \le   2\biggl[ 1 + \sum_{1 \le i < j \le p+1} (k_i - k_{i-1})(k_j - k_{j-1}) \biggr],
\]
a simple consequence of $\sum_{i=1}^{p+1} (k_i - k_{i-1}) = n$.

Our bound is always better than G\"{u}nther's bound in \cite{91Gunther}, i.e.,
\[
\frac12 (n-1)(n+2) < \max \Bigl\{ \frac{m(m+3)}{2} + 5,  \frac{m(m+5)}{2} \Bigr\},
\]
where $m$ is as in \eqref{eq:m}. This follows from $m \ge n-1$ and thus $m(m+3) \ge (n-1)(n+2)$.

Our bound is also better than Wang's bound in \cite{Wang22} for all finite subgroups $G \subseteq \SO_n(\mathbb{R})$ of cardinality $|G| > (n-1)(n+2)/4m$, e.g., the alternating group on $n$ letters or cyclic group of order $q > (n-1)(n+2)/4m$. However the real coup here is that our bound holds for any subgroup $G \subseteq \SO_n(\mathbb{R})$ and not just finite ones: Indeed if
\begin{gather*}
d_G \coloneqq \min \{ d: \Flag(n_1,\dots,  n_p, \mathbb{R}^n)\text{ can be $G$-equivariantly embedded in }\mathbb{R}^d \},
\shortintertext{then}
d_G \le \frac12 (n-1)(n+2) = d_{\SO_n(\mathbb{R})} = d_\MP.
\end{gather*}
To the best of our knowledge, effective bounds like Wang's \cite{Wang22} have never before been established for infinite groups.

\section{Notations and some background}\label{sec:rep}

The real vector spaces of real $m \times n$ matrices, symmetric, skew-symmetric, and traceless symmetric $n \times n$ matrices are denoted by $\mathbb{R}^{m \times n}$, $\Sym^2(\mathbb{R}^n)$, $\Alt^2(\mathbb{R}^n)$, $\Sym^2_\oh(\mathbb{R}^n)$ respectively.  We reserve the letter $\mathbb{W}$ for $\mathbb{R}$-vector spaces, $\mathbb{U}$ for $\SO_n(\mathbb{C})$-modules, and $\mathbb{V}$ for real $\SO_n(\mathbb{R})$-modules, usually adorned with various subscripts. Let $n_1 + \dots + n_p = n$. We write
\begin{multline*}
\S(\O_{n_1}(\mathbb{R}) \times \cdots \times \O_{n_p}(\mathbb{R}) ) \coloneqq \{ \diag(X_1,\dots,X_p) \in \O_n(\mathbb{R}) :\\
 X_1 \in \O_{n_1}(\mathbb{R}),\dots, X_p \in \O_{n_p}(\mathbb{R}), \; \det( X_1) \cdots \det(X_p) = 1 \}.
\end{multline*}

We recall some pertinent facts from representation theory for easy reference. Irreducible $\SO_n(\mathbb{C})$-modules are indexed by nonincreasing half-integer sequences $\mu = (\mu_1,\dots, \mu_m)$ where $m \coloneqq \lfloor n/2 \rfloor$, $\mu_m \ge 0$ if $n = 2m + 1$, and $\mu_{m-1} \ge \lvert \mu_m \rvert$ if $n = 2m$  \cite[Proposition~3.1.19]{GW09}. Let $\mathbb{U}_\mu$ be the irreducible $\SO_n(\mathbb{C})$-module indexed by $\mu$. Then its dimension is given by
\begin{equation}\label{eq:dimension}
\dim \mathbb{U}_\mu = \begin{cases}
\prod\limits_{1 \le i < j \le m} \dfrac{\mu_i - \mu_j - i +  j}{j - i} \prod\limits_{1 \le i \le j \le m} \dfrac{\mu_i + \mu_j + n - i -  j}{n - i- j}& \text{if } n = 2m+1, \\[4ex]
\prod\limits_{1 \le i < j \le m} \dfrac{\mu_i - \mu_j - i +  j}{j - i} \dfrac{\mu_i + \mu_j + n - i -  j}{n - i- j} &\text{if } n = 2m.
\end{cases}
\end{equation}
The \emph{fundamental weights} of $\SO_n(\mathbb{C})$ are $\omega_1,\dots, \omega_m \in \mathbb{R}^m$ defined by 
\begin{equation}\label{eq:fundweight}
\omega_i = \begin{cases}
e_1 + \dots + e_i &\text{if } n = 2m+1\text{ and }1\le i\le m-1, \\[1ex]
\frac{1}{2} (e_1 + \dots + e_m) &\text{if } n = 2m+1\text{ and }i = m, \\[1ex]
e_1 + \dots + e_i &\text{if } n = 2m\text{ and }1\le i\le m-2, \\[1ex]
\frac{1}{2} (e_1 + \dots + e_{m-1} - e_m) &\text{if } n = 2m\text{ and }i = m-1, \\[1ex]
\frac{1}{2} (e_1 + \dots + e_{m-1} + e_m) &\text{if } n = 2m\text{ and }i = m,
\end{cases}
\end{equation}
where $e_1,\dots, e_m\in \mathbb{R}^m$ are the standard basis vectors. Irreducible $\SO_n(\mathbb{R})$-modules may be determined from irreducible $\SO_n(\mathbb{C})$-modules in the following situation \cite[Propositions~26.26 and 26.27]{FH91}, which is all we need.
\begin{proposition}\label{prop:irrep}
Let $n \in \mathbb{N}$ and $m = \lfloor n/2 \rfloor$. 
If $\mu_m = 0$ for $n = 2m + 1$ or $\mu_{m-1} = \mu_m = 0$ for $n = 2m$,  then every irreducible $\SO_n(\mathbb{C})$-module takes the form $\mathbb{V}_\mu \otimes \mathbb{C}$ for some irreducible $\SO_n(\mathbb{R})$-module $\mathbb{V}_\mu$.  In particular,  $\dim_{\mathbb{R}} \mathbb{V}_\mu = \dim_{\mathbb{C}} \mathbb{V}_\mu \otimes \mathbb{C}$.
\end{proposition}
The following dimension formulas may be calculated from \eqref{eq:dimension} and \eqref{eq:fundweight} \cite[Propositions~20.15 and 20.20]{FH91}:
\begin{equation}\label{eq:irrep2}
\begin{alignedat}{3}
\dim_{\mathbb{C}} \mathbb{U}_{\omega_m} &= 2^m \qquad &\text{if }n &= 2m + 1;\\
\dim_{\mathbb{C}} \mathbb{U}_{\omega_{m-1}} = \dim_{\mathbb{C}} \mathbb{U}_{\omega_m} &= 2^{m-1} \qquad &\text{if } n &= 2m.
\end{alignedat}
\end{equation}

\section{Minimal equivariant embeddings of a flag manifold}\label{sec:em}

In this and the next sections, we frame all our discussions in terms of the flag manifold, and then deduce the corresponding Grassmannian results as the $p = 1$ special case.

The two goals of  this section are to prove Proposition~\ref{prop:flagV2} and Theorem~\ref{thm:minimal flag}. Collectively they show that if $\varepsilon: \Flag(k_1,\dots,k_p, \mathbb{R}^n) \to \mathbb{V}$ is any $\SO_n(\mathbb{R})$-equivariant embedding of the lowest dimension, then $\mathbb{V}$ has to be isomorphic to $\Sym^2_\oh (\mathbb{R}^n)$, the space of traceless symmetric matrices. In particular it is not possible to equivariantly embed $\Flag(k_1,\dots,k_p, \mathbb{R}^n)$ into any vector space of dimension less than $\dim \Sym^2_\oh (\mathbb{R}^n)  = \frac12 (n-1)(n+2)$. The isospectral model \eqref{eq:iso} will be seen to be the image of $\varepsilon$ in Proposition~\ref{prop:min}.

We first define the notion of equivariant embedding, which has been extensively studied in a more general setting \cite{Mostow57,  Bredon72,  LV83, Timashev11}.
\begin{definition}[Equivariant embedding and equivariant submanifold]
Let $G$ be a group, $\mathbb{V}$ be a $G$-module, and $\mathcal{M}$ be a $G$-manifold.  An embedding $\varepsilon: \mathcal{M} \to \mathbb{V}$ is called a $G$-equivariant embedding if $\varepsilon(g \cdot x) = g \cdot \varepsilon(x)$ for all $x\in \mathcal{M}$ and $g \in G$.  In this case, the embedded image $\varepsilon(\mathcal{M})$ is called a $G$-submanifold of $\mathbb{V}$.
\end{definition}
In our context $G = \SO_n(\mathbb{R})$, $\mathcal{M} = \Flag(k_1,\dots,k_p, \mathbb{R}^n)$, $\mathbb{V} = \mathbb{R}^{n \times n}$, $G$ acts on $\mathcal{M}$ through
\[
X \cdot (\mathbb{W}_1 \subseteq \dots \subseteq \mathbb{W}_p ) \coloneqq (X \cdot \mathbb{W}_1 \subseteq \dots \subseteq  X \cdot \mathbb{W}_p),
\]
and on $\mathbb{V}$ by conjugation,
\[
\SO_n(\mathbb{R}) \times  \mathbb{R}^{n \times n} \to  \mathbb{R}^{n \times n}, \quad (Q,  X) \mapsto Q \cdot X \coloneqq Q X Q^\tp.
\]
In this case the decomposition of the $\SO_n(\mathbb{R})$-module $\mathbb{R}^{n \times n}$ into irreducible $\SO_n(\mathbb{R})$-submodules is very simple, given by
\[
\mathbb{R}^{n\times n} = \mathbb{R}I \oplus \Alt^2(\mathbb{R}^n) \oplus \Sym^2_\oh (\mathbb{R}^n),
\]
where $\mathbb{R} I \coloneqq \{\gamma I \in \mathbb{R}^{n\times n}:  \gamma \in \mathbb{R}\}$ denotes the space of constant diagonal matrices. In the notations of Section~\ref{sec:rep},
\[
\mathbb{R}^{n\times n} \simeq \mathbb{V}_{0,\dots, 0} \oplus \mathbb{V}_{1,1,0\dots, 0} \oplus \mathbb{V}_{2,0\dots, 0},
\]
where the irreducible $\SO_n(\mathbb{R})$-submodules are indexed by a non-increasing sequence of half-integers; note that we write $\simeq$ as these are only determined up to $\SO_n(\mathbb{R})$-module isomorphisms. So
\[
\mathbb{V}_{0,\dots, 0} \simeq \mathbb{R} I, \quad
\mathbb{V}_{1,1,0\dots, 0} \simeq \Alt^2(\mathbb{R}^n), \quad \mathbb{V}_{2,0\dots, 0} \simeq \Sym^2_\oh (\mathbb{R}^n).
\]
Strictly speaking we should use the notations on the left when referring to these spaces as $\SO_n(\mathbb{R})$-modules and those on the right when referring to them as vector spaces.

There is an important fourth irreducible $\SO_n(\mathbb{R})$-module not among those that appeared so far. Consider the action of $\SO_n(\mathbb{R})$ on $\mathbb{R}^n$ through matrix-vector product,
\[
\SO_n(\mathbb{R}) \times \mathbb{R}^n \to \mathbb{R}^n,\quad (Q,  v) \mapsto Q\cdot v \coloneqq Qv.
\]
This turns $\mathbb{R}^n$ into an irreducible $\SO_n(\mathbb{R})$-module and when regarded as such it is denoted $\mathbb{V}_{1,0\dots, 0} $. Collectively these four irreducible $\SO_n(\mathbb{R})$-modules are the ones that matter most to us, with dimensions
\[
\dim \mathbb{V}_{0,\dots, 0} = 1, \quad \dim \mathbb{V}_{1,0,\dots, 0} = n, \quad \dim \mathbb{V}_{1,1,0\dots, 0} = \frac12 n(n-1),\quad \dim \mathbb{V}_{2,0\dots, 0} = \frac12 (n-1)(n+2).
\]
Note that we have adopted the standard convention that a trailing sequence of zeros in the indices contains as many copies as is necessary to pad the entire sequence of indices to its requisite length.

We begin by showing that the flag manifold may be equivariantly embedded in $\Sym^2_\oh (\mathbb{R}^n)$.
\begin{proposition}[Flag manifold as equivariant submanifold]\label{prop:flagV2}
Let $0 \eqqcolon k_0  <   k_1 < \cdots <  k_p  < k_{p+1} \coloneqq n$ be integers.  Then $\Flag(k_1,\dots,  k_p,\mathbb{R}^n)$ may be embedded as an $\SO_n(\mathbb{R})$-submanifold of $\mathbb{V}_{2,0\dots, 0} \simeq \Sym^2_\oh (\mathbb{R}^n)$.
\end{proposition}
\begin{proof}
Let $0 \eqqcolon k_0  <   k_1 < \cdots <  k_p  < k_{p+1} \coloneqq n$ be integers and
\[
n_{i+1} = k_{i+1} - k_i,\quad i =1,\dots, p.
\]
Let $a_1 >  \cdots > a_{p} > 0$ be real numbers and 
\[
a_{p+1} \coloneqq - \frac{n_1  a_1 + \dots + n_p a_p}{n_{p+1}}.
\]
Write $\alpha  \coloneqq (a_1,\dots, a_{p+1})$, $\mu \coloneqq (n_1,\dots,  n_{p+1})$, and consider the block diagonal matrix 
\[
I_{\alpha ,\mu} = \begin{bmatrix}
a_1 I_{n_1} & 0 & \cdots & 0 \\
0 &  a_2 I_{n_2} & \cdots & 0 \\
\vdots & \vdots & \ddots & \vdots \\
0 &  0 & \cdots &  a_{p+1} I_{n_{p+1}}
\end{bmatrix} \in \mathbb{R}^{n\times n}.
\]
Since $\tr(I_{\alpha ,\mu}) = 0$, we have $I_{\alpha ,\mu} \in \Sym^2_\oh (\mathbb{R}^n)$.  Since $\SO_n(\mathbb{R})$ acts on $\Sym^2_\oh (\mathbb{R}^n)$  by conjugation and the stablizer of $I_{\alpha ,\mu}$ is $\S(\O_{n_1}(\mathbb{R}) \times \cdots \times \O_{n_{p+1}}(\mathbb{R}))$, we have 
\begin{align*}
 \lbrace Q I_{\alpha ,\mu} Q^\tp \in \Sym^2_\oh (\mathbb{R}^n) :  Q\in \SO_n(\mathbb{R})  \rbrace &\cong \SO_{n}(\mathbb{R})/ \S(\O_{n_1}(\mathbb{R}) \times \cdots \times \O_{n_{p+1}}(\mathbb{R})) \\
&\cong \Flag(k_1,\dots,  k_p, \mathbb{R}^n)
\end{align*}
by \eqref{eq:diffeo}. Since the first set is contained in $\Sym^2_\oh (\mathbb{R}^n)$, we obtain an embedding $\varepsilon$ of $\Flag(k_1,\dots,  k_p, n)$ into $\Sym^2_\oh (\mathbb{R}^n)$.  The $\SO_n(\mathbb{R})$-equivariance of $\varepsilon$ follows from that of $\varphi_1,\varphi_3,\theta$, a routine verification.
\end{proof}

The requirement that $n \ge 17$ for our next result is to ensure that $2^{\lfloor (n-1) /2\rfloor} > \frac12 (n-1)(n+2)$. As the remaining results in this section all depend on the next one, this mild requirement will propagate to those results too.
\begin{lemma}\label{lem:lowdim0}
Let $n \in \mathbb{N}$ with $n\ge 17$.  Let $\mathbb{V}$ be an irreducible $\SO_n(\mathbb{R})$-module with $\dim_{\mathbb{R}} \mathbb{V} \le \frac12 (n-1)(n+2)$.  Then $\mathbb{V} = \mathbb{V}_\mu$ for some
\begin{equation}\label{eq:mu}
\mu = \begin{cases}
 (\mu_1,  \dots, \mu_{m-1}, 0) \in \mathbb{N}^m &\text{~if~}n =2m + 1,  \\
 (\mu_1,  \dots, \mu_{m-2}, 0,0) \in \mathbb{N}^m &\text{~if~}n =2m.
 \end{cases}
\end{equation}
\end{lemma}
\begin{proof}
Note that $\mathbb{V} \otimes \mathbb{C}$ is an $\SO_n(\mathbb{C})$-module, but it may not be irreducible.  Let
\[
\mathbb{V} \otimes \mathbb{C} =\bigoplus_{\mu} \mathbb{U}_\mu^{\oplus p_\mu}
\]
be its decomposition into  irreducible  $\SO_n(\mathbb{C})$-submodules $\mathbb{U}_\mu$ as described in Section~\ref{sec:rep}, with $p_\mu$ the multiplicity of $\mathbb{U}_\mu$. For any $\mathbb{U}_\mu$ in the direct summands above, we must have $\dim_{\mathbb{C}} \mathbb{U}_\mu \le \frac12 (n-1)(n+2)$.

We may express the half-integer sequence $\mu$ as a nonnegative integer combination of fundamental weights $\mu = c_1 \omega_1 + \dots + c_m \omega_m$.  If $n = 2m + 1 \ge 17$ and $c_m > 0$, then 
\[
\dim_{\mathbb{C}} \mathbb{U}_\mu \ge \dim_{\mathbb{C}} \mathbb{U}_{\omega_m} = 2^{m} > \frac12 (n-1)(n+2),
\]
where the equality is by \eqref{eq:irrep2} and the last inequality is the reason for requiring $n \ge 17$.  Similarly, if $n = 2m \ge 17$ and $c_m > 0$ or $c_{m-1} > 0$,  then
\[
\dim_{\mathbb{C}} \mathbb{U}_\mu \ge \dim_{\mathbb{C}} \mathbb{U}_{\omega_m}= \dim_{\mathbb{C}} \mathbb{U}_{\omega_{m-1}} = 2^{m-1} > \frac12 (n-1)(n+2)
\]
where again the equality is by \eqref{eq:irrep2}.
In either case, we have a contradiction to $\dim_{\mathbb{C}} \mathbb{U}_\mu \le \frac12 (n-1)(n+2)$.  Hence we obtain
\[
\mu = \begin{cases}
c_1 \omega_1 + \dots + c_{m-1} \omega_{m-1} &\text{if } n =2m + 1,  \\
c_1 \omega_1 + \dots + c_{m-2} \omega_{m-2} &\text{if } n =2m,
 \end{cases}
\]
and therefore \eqref{eq:mu}, with an application of \eqref{eq:fundweight}.
It follows from Proposition~\ref{prop:irrep} that $\mathbb{U}_\mu = \mathbb{V}_\mu \otimes \mathbb{C}$ for some irreducible $\SO_n(\mathbb{R})$-submodule $\mathbb{V}_\mu$.  Since $\mathbb{V}$ is also irreducible, we must have $\mathbb{V} = \mathbb{V}_\mu$. 
\end{proof}

The next lemma is a key ingredient for our main result of this section. In its proof, we adopt the standard shorthand for repeated ones in the indices: $1^q \coloneqq \underbrace{1,\dots,  1}_{q\text{ copies}}$ for any $q \in \mathbb{N}$.
\begin{lemma}\label{lem:lowdim}
Let $n \in \mathbb{N}$ with $n\ge 17$. Then the only irreducible $\SO_n(\mathbb{R})$-modules of dimensions strictly smaller than $\frac12 (n-1)(n+2)$ are
\[
\mathbb{V}_{0,\dots, 0} \simeq \mathbb{R},  \qquad \mathbb{V}_{1,0,\dots, 0} \simeq \mathbb{R}^n, \qquad \mathbb{V}_{1,1,0\dots, 0} \simeq \Alt^2(\mathbb{R}^n);
\]
and the only irreducible $\SO_n(\mathbb{R})$-module of dimension exactly $\frac{1}{2}(n-1)(n+2)$ is
\[
\mathbb{V}_{2,0,\dots, 0} \simeq \mathsf{S}^2_\oh(\mathbb{R}^n).
\]
\end{lemma}
\begin{proof}
Let $n \in \mathbb{N}$ and $m \coloneqq \lfloor n/2 \rfloor$. Let $f_n: \{ x\in \mathbb{R}^m: x_1 \ge \cdots \ge x_m \} \to \mathbb{R}$ be defined by
\[
f_n(x_1,\dots, x_m) = \begin{cases}
\prod\limits_{1 \le i < j \le m} \dfrac{x_i - x_j - i +  j}{j - i} \prod\limits_{1 \le i \le j \le m} \dfrac{x_i + x_j + n - i -  j}{n - i- j} &\text{if } n = 2m+1, \\[4ex]
\prod\limits_{1 \le i < j \le m} \dfrac{x_i - x_j - i +  j}{j - i} \dfrac{x_i + x_j + n - i -  j}{n - i- j} &\text{if } n = 2m.
\end{cases}
\]
So for any non-increasing sequence of half-integers $\mu = (\mu_1,\dots, \mu_m)$,  $f_n(\mu)$ is exactly the expression in  \eqref{eq:dimension}. Therefore $f_n(\mu) = \dim_{\mathbb{C}} \mathbb {U}_{\mu}$ for any irreducible $\SO_n(\mathbb{C})$-submodule $\mathbb{U}_\mu$.  Let $k$ be such that $x_{k+1} = \dots = x_{m} = 0 < x_k$. For any $\delta \in \mathbb{R}$ with $0 < \delta < x_k$, let $y \in \mathbb{R}^n$ be given by
\[
y = (x_1 - \delta ,\dots, x_k - \delta,  0 ,\dots,  0).
\]
By its definition, $f_n(x)$ only depends on $x_i - x_j$ and $x_i + x_j$ for any $i,j \in \{1,\dots, m\}$. So
\[
y_i - y_j \le x_i - x_j,\quad y_i 
+ y_j \le x_i + x_j
\]
for any $i,j \in \{1,\dots, m\}$ and strict inequality must hold for some $(i,j)$. Hence $f_n (y) < f_n(x)$. 

Let $\mathbb{V}$ be an irreducible $\SO_n(\mathbb{R})$-module of dimension at most $\frac{1}{2} (n-1)(n+2)$.  By Lemma~\ref{lem:lowdim0},  $\mathbb{V} = \mathbb{V}_{\mu}$ for some $\mu$ as in \eqref{eq:mu}. Proposition~\ref{prop:irrep} assures that $f_n(\mu) = \dim_{\mathbb{R}} \mathbb{V}_\mu$.  In particular,  if $\mu = (\mu_1, \dots, \mu_k,  0,\dots, 0)$ and $\mu_k > 0$,  then
\begin{equation}\label{lem:lowdim:eq1}
\dim \mathbb{V}_{\mu} = f_n(\mu) > f_n(\mu_1-1, \dots, \mu_k-1,  0,\dots, 0) =  \dim \mathbb{V}_{\mu_1-1, \dots, \mu_k-1,  0,\dots, 0}.
\end{equation}

As $\dim \mathbb{V}_{2,0,\dots, 0} = \frac12 (n-1)(n+2)$, if $\mu_1  - \mu_2 \ge 2$, then  $\dim \mathbb{V}_\mu \ge \dim \mathbb{V}_{2,0,\dots, 0}$ by applying \eqref{lem:lowdim:eq1} recursively.  We claim that equality holds if only if $\mu = (2,0,\dots,  0)$. Suppose not and we have $\mu_1 \ge 3$. If $\mu_2 \ge 1$, then \eqref{lem:lowdim:eq1} applies to give  $\dim \mathbb{V}_\mu > \frac12 (n-1)(n+2)$. If $\mu_2 = 0$, then $\mu = (\mu_1, 0 ,\dots,  0)$ and by \eqref{eq:dimension},
\[
\dim \mathbb{V}_\mu = \dim \mathbb{V}_{\mu_1, 0 ,\dots,  0} = \begin{cases} 
\biggl(1 +  \dfrac{2\mu_1}{n-2} \biggr) \dbinom{n - 3 + \mu_1}{\mu_1} &\text{if } n = 2m+1, \\[4ex]
\biggl( 1 + \dfrac{\mu_1}{m-1} \biggr) \dbinom{n  -3 + \mu_1 }{\mu_1} &\text{if } n = 2m.
\end{cases}
\]
In either case we  get $\dim \mathbb{V}_\mu > \frac12 (n-1)(n+2)$ as $\mu_1 \ge 3$ and $n \ge 17$.  Therefore if $\dim \mathbb{V}_\mu = \frac12 (n-1)(n+2)$, then $\mu =( 2,0,\dots, 0)$.

If $\dim \mathbb{V}_\mu < \dim \mathbb{V}_{2,0,\dots, 0}$, then $\mu_1  - \mu _2 = 1$ or $0$.  We consider the two cases below.

\underline{\textsc{Case I.} $\mu_1  - \mu _2 = 1$}:  In this case $\mu = (\mu_1,  \mu_1 - 1,  \mu_3 , \dots, \mu_m)$.  If $\mu_1 -  \mu_3 \ge 2$, we have a contradiction:
\[
\dim \mathbb{V}_\mu  \ge \dim \mathbb{V}_{2,  1,  0,\dots,  0} > \dim \mathbb{V}_{2,  0,\dots,  0}.
\]
As before, the first inequality comes from applying \eqref{lem:lowdim:eq1} recursively and the second is by a direct calculation --- all similar inequalities below are obtained using this same recipe. If $\mu_1 - \mu_3 = 1$ and $\mu_3 \ge \mu_4 + 1$, we have a contradiction:
\[
\dim \mathbb{V}_\mu  \ge \dim \mathbb{V}_{\mu_1 - \mu_3 + 1,  \mu_1 - \mu_3,  1,0 , \dots,  0} = \dim \mathbb{V}_{2,  1,  1,0 , \dots,  0} > \dim \mathbb{V}_{2,  0,\dots,  0}.
\]
If $\mu_1 - \mu_3 = 1$ and $\mu_3 = \cdots = \mu_q > \mu_{q+1}$ for some $m \ge q \ge 4$, then $\mu = (1,0,\dots, 0)$ or otherwise we have a contradiction:
\[
\dim \mathbb{V}_\mu \ge \dim \mathbb{V}_{\mu_1 - \mu_q + 1,  \mu_1 - \mu_q,   1^{q-2},0, \dots, 0} = \dim \mathbb{V}_{2,  1^{q-1}, 0, \dots, 0} > \dim \mathbb{V}_{2,  0,\dots,  0}.
\]

\underline{\textsc{Case II.} $\mu_1 - \mu_2 = 0$}:  In this case $\mu = (\mu_1,  \mu_1,  \mu_3,\dots,  \mu_m)$.  If $\mu_1 - \mu_3 \ge 1$ and $\mu_3 - \mu_4 \ge 1$, we have a contradiction:
\[
\dim \mathbb{V}_\mu \ge \dim \mathbb{V}_{\mu_1 - \mu_4,  \mu_1 - \mu_4,  \mu_3 - \mu_4,  0, \dots, 0} 
\ge \dim \mathbb{V}_{\mu_1 - \mu_3 +1,  \mu_1 - \mu_3 + 1,  1, 0,\dots, 0} 
\ge \dim \mathbb{V}_{2,0,\dots, 0}.
\]
If $\mu_1 - \mu_3 \ge 1$ and $\mu_3 = \mu_q > \mu_{q+1}$ for some $m \ge q \ge 4$, then $\mu = (1,1,0,\dots,  0)$ or otherwise $\mu_1 \ge 2$ and we have a contradiction:
\begin{align*}
\dim \mathbb{V}_\mu &\ge \dim \mathbb{V}_{\mu_1 - \mu_{q+1},  \mu_1 - \mu_{q+1} ,\mu_q - \mu_{q+1}, \dots,  \mu_q - \mu_{q+1}, 0,\dots, 0} \\
&\ge \dim \mathbb{V}_{\mu_1 - \mu_q + 1,  \mu_1 - \mu_q + 1 ,1^{q-2}, 0,\dots, 0} \ge \dim \mathbb{V}_{2,0,\dots,  0}.
\end{align*}
If $\mu_1 =  \mu_q > \mu_{q+1}$ for some $m \ge q \ge 3$, then $\mu = (0,\dots,  0)$ or otherwise we have a contradiction: 
\[
\dim \mathbb{V}_\mu \ge \dim \mathbb{V}_{\mu_1 - \mu_{q+1},  \dots,  \mu_q - \mu_{q+1}, 0 ,\dots, 0} \ge \dim \mathbb{V}_{1^q, 0,\dots,  0} > \dim \mathbb{V}_{2,0,\dots,  0}.
\]

Hence the only possibilities are $\mathbb{V}_\mu = \mathbb{V}_{0,\dots, 0}$,  $\mathbb{V}_{1,0,\dots, 0}$, or $\mathbb{V}_{1,1,0,\dots, 0}$.  
\end{proof}

We now show that with the possible exceptions of some low-dimensional cases, the $\SO_n(\mathbb{R})$-equivariant embedding in Proposition~\ref{prop:flagV2} has the lowest possible ambient dimension.
\begin{theorem}[Minimal equivariant embedding of flag manifolds]\label{thm:minimal flag}
Let $0 \eqqcolon k_0  <   k_1 < \cdots <  k_p  < k_{p+1} \coloneqq n$ be integers with $n \ge 17$. Let $\mathbb{V}$ be an $\SO_n(\mathbb{R})$-module of the smallest dimension such that $\Flag(k_1,\dots,  k_p,\mathbb{R}^n)$ is an $\SO_n(\mathbb{R})$-submanifold of $\mathbb{V}$. Then $\mathbb{V}  \simeq \Sym^2_\oh (\mathbb{R}^n)$.
\end{theorem}
\begin{proof}
By Proposition~\ref{prop:flagV2}, we know that  $\mathbb{V}_{2,0,\dots,0}  \simeq \Sym^2_\oh (\mathbb{R}^n)$ serves as an ambient space of  dimension $\frac12 (n-1)(n+2)$ in which  $\Flag(k_1,\dots,  k_p,\mathbb{R}^n)$ equivariantly embeds. Let  $\varepsilon: \Flag(k_1,\dots,  k_p, n) \to \mathbb{V}$ be an $\SO_n(\mathbb{R})$-equivariant embedding where $\mathbb{V}$ has the lowest possible dimension and $ \mathbb{V} \not\simeq \mathbb{V}_{2,0,\dots,  0}$.  We will derive a contraction.

By Proposition~\ref{prop:flagV2}, we must have $\dim \mathbb{V} \le \dim \mathbb{V}_{2,0,\dots,0} = \frac12 (n-1)(n+2)$. We have assumed that $ \mathbb{V} \not\simeq \mathbb{V}_{2,0,\dots,  0}$; so by Lemma~\ref{lem:lowdim}, a decomposition of $\mathbb{V}$ into irreducible $\SO_n(\mathbb{R})$-submodules must take the form
\[
\mathbb{V} = \mathbb{V}_{0,\dots, 0}^{\oplus a} \oplus  \mathbb{V}_{1,0\dots, 0}^{\oplus b}  \oplus \mathbb{V}_{1,1,0\dots, 0}^{\oplus c},
\]
where $a,b,c \in \mathbb{N}$ satisfy
\[
a + b n + c \binom{n}{2} \le \frac12 (n-1)(n+2).
\]

We claim that $a = 0$. Consider the projection $\pi_0: \mathbb{V} \to \mathbb{V}_{0,\dots, 0}^{\oplus a}$.  Since $\varepsilon$ and $\pi_0$ are $\SO_n(\mathbb{R})$-equivariant,  $\SO_n(\mathbb{R})$ acts on $\Flag(k_1,\dots,  k_p,\mathbb{R}^n)$ transitively, and $\mathbb{V}_{0,\dots, 0}$ is the trivial representation,  $\pi_0 \circ \varepsilon$ must be a constant map.  Thus $\varepsilon - \pi_0 \circ \varepsilon$ is an equivariant embedding of $\Flag(k_1,\dots,  k_p,\mathbb{R}^n)$ into  the  ambient space $\mathbb{V}_{1,0\dots, 0}^{\oplus b} \oplus \mathbb{V}_{1,1,0\dots, 0}^{\oplus c}$, which has dimension  lower than $\mathbb{V}$ unless $a = 0$.

We must also have $c \le 1$ since 
\[
c \le \frac{(n-1)(n+2)}{n(n-1)} = \frac{n+2}{n} < 2.
\]
If $c = 1$,  consider the map $\pi_{1,1} \circ \varepsilon: \Flag(k_1,\dots,  k_p,n) \to \mathbb{V}_{1,1,0\dots, 0}$ where $\pi_{1,1}: \mathbb{V} \to \mathbb{V}_{1,1,0\dots, 0}$ is the projection.  Since $\pi_{1,1} \circ \varepsilon$ is $\SO_n(\mathbb{R})$-equivariant and $\SO_n(\mathbb{R})$ acts on $\Flag(k_1,\dots,  k_p,\mathbb{R}^n)$ transitively,  for any $(\mathbb{W}_1 \subseteq \cdots \subseteq \mathbb{W}_p) \in \Flag(k_1,\dots,  k_p,\mathbb{R}^n)$, we must have
\[
\pi_{1,1} \circ \varepsilon ( \Flag(k_1,\dots,  k_p,\mathbb{R}^n) ) = \{Q A_0 Q^\tp: Q\in \SO_n(\mathbb{R})\},
\]
where $A_0 \coloneqq \pi_{1,1} \circ \varepsilon (\mathbb{W}_1 \subseteq \cdots \subseteq \mathbb{W}_p) \in \mathbb{V}_{1,1,0\dots, 0}$.  If $A_0 = 0$, then $\pi_{1,1}\circ \varepsilon$ is the zero map, leading to the conclusion that $\varepsilon$ is an $\SO_n(\mathbb{R})$-equivariant embedding of $\Flag(k_1,\dots,  k_p,\mathbb{R}^n)$ into $\mathbb{V}_{1,0\dots, 0}^{\oplus b}$, and thereby contradicting minimality of $\mathbb{V}$.  Hence $A_0 \ne 0$. In which case
\[
\Stab(A_0) \coloneqq \lbrace
Q\in \SO_n(\mathbb{R}): Q A_0 Q^\tp = A_0
\rbrace \supseteq \S(\O_{n_1}(\mathbb{R}) \times \cdots \times \O_{n_{p+1}}(\mathbb{R}))
\]
and so the identity component of $\Stab(A_0)$,
\[
G_0 \supseteq \SO_{n_1}(\mathbb{R}) \times \cdots \times \SO_{n_{p+1}}(\mathbb{R}).
\]
We may assume that $A_0$ is in its normal form  \cite{Thompson88}
\begin{align*}
A_0 &= \begin{bmatrix}
J_1 & 0 & \cdots &0  \\
0 & J_2 & \cdots & 0  \\
\vdots & \vdots & \ddots & \vdots \\
0 & 0 & \cdots & J_{p+1} 
\end{bmatrix} \in \mathbb{R}^{n \times n},
\intertext{where for $i =1,\dots,p$,}
J_i &= \begin{bmatrix}
0 & \lambda_i I_{m_i} \\
-\lambda_i I_{m_i} & 0
\end{bmatrix} \in \mathbb{R}^{2m_i \times 2m_i},
\end{align*}
with $\lambda_1 > \cdots > \lambda_p > 0$ and $m_1,\dots,  m_p \in \mathbb{N}$ satisfying
\[
2 m_1 + \cdots  + 2 m_p = \rank(A_0) \eqqcolon r.
\]
If $n > r$, then the bottom right block $J_{p+1} \in \mathbb{R}^{(n-r) \times (n-r)}$ is set to be the zero matrix. 

A direct calculation shows that 
\[
G_0 =\biggl\lbrace \begin{bmatrix}
I_r & 0 \\
0 & V
\end{bmatrix}: V \in \SO_{n-r}(\mathbb{R})\biggr\rbrace \supseteq \SO_{n_1}(\mathbb{R}) \times \cdots \times \SO_{n_{p+1}}(\mathbb{R})
\]
if and only if 
\[
p+1 \ge r +1,\qquad n_1 = \cdots = n_r = 1.
\]
Now observe that for any $V\in \O_{n_{r + 1}}(\mathbb{R})$ with $\det(V) = - 1$,  the matrix
\[
 \begin{bmatrix}
I_{r - 2} & 0 & 0 & 0 & 0 \\
0& 1 &  0 & 0 & 0 \\
0 & 0 & -1 & 0 & 0 \\
0 & 0 & 0 & V  & 0 \\ 
0 & 0 & 0 & 0 & I_{n - r - n_{r + 1}} \\
\end{bmatrix} \in \S(\underbrace{\O_{1}(\mathbb{R})\times \cdots \times \O_{1}(\mathbb{R})}_{r\text{~copies}} \times \O_{n_{r + 1}}(\mathbb{R}) \times \cdots \times \O_{n_{p+1}}(\mathbb{R}))
\]
but is not in $\Stab(A_0)$. This yields the required contradiction. 

It remains to consider the case $\mathbb{V} = \mathbb{V}_{1,0\dots, 0}^{\oplus b} \simeq \mathbb{R}^{n\times b}$. The action here is $\SO_{n}(\mathbb{R}) \times \mathbb{R}^{n \times k} \to \mathbb{R}^{n \times k}$, $(Q, Y) \mapsto QY$. Since $\varepsilon$ is $\SO_n(\mathbb{R})$-equivariant,  $0 \not\in \varepsilon (\Flag(k_1,\dots,  k_p,\mathbb{R}^n))$.  Given a nonzero $Y \in \mathbb{R}^{n\times b}$ of rank $q >0$,
\[
\Stab(Y) =\lbrace
Q\in \SO_n(\mathbb{R}):  QY = Y
\rbrace \simeq \SO_{n-q}(\mathbb{R}),
\]
from which it follows that $\S(\O_{n_1}(\mathbb{R}) \times \cdots \times \O_{n_{p+1}}(\mathbb{R})) \not\subseteq \Stab(Y)$, a contradiction.
\end{proof}
In general, the isospectral model in \eqref{eq:iso} is a submanifold of $\Sym^2 (\mathbb{R}^n)$ but it can always be moved into  $\Sym^2_\oh (\mathbb{R}^n)$ by a translation. We will call an isospectral model that sits inside  $\Sym^2_\oh (\mathbb{R}^n)$ a \emph{traceless isospectral model}. In Proposition~\ref{prop:min}, we will see that they give exactly the minimal equivariant embedding in Theorem~\ref{thm:minimal flag}. 
\begin{proposition}[Traceless isospectral model]\label{prop:classification}
If $\varepsilon: \Flag(k_1,\dots,  k_p,\mathbb{R}^n) \to \Sym^2 (\mathbb{R}^n)$ is an $\SO_n(\mathbb{R})$-equivariant embedding,  then there exist an $\SO_n(\mathbb{R})$-equivariant embedding $\varepsilon_{\circ}: \Flag(k_1,\dots,  k_p,\mathbb{R}^n) \to \Sym^2_\oh (\mathbb{R}^n)$ and  $c\in \mathbb{R}$ such that  $\varepsilon = \varepsilon_{\circ} + c$, i.e.,
\[
\varepsilon (\mathbb{W}_1 \subseteq \cdots \subseteq \mathbb{W}_p) = \varepsilon_{\circ}(\mathbb{W}_1 \subseteq \cdots \subseteq \mathbb{W}_p) +  c I
\]
for any $(\mathbb{W}_1 \subseteq \cdots \subseteq \mathbb{W}_p)\in \Flag(k_1,\dots,  k_p,\mathbb{R}^n)$.
\end{proposition}
\begin{proof}
The decomposition of  $\Sym^2 (\mathbb{R}^n) $ into irreducible $\SO_n(\mathbb{R})$-submodules is simply
\begin{equation}\label{eq:irrep2n}
\Sym^2 (\mathbb{R}^n) = \mathbb{R} I \oplus \Sym_{\oh}^2 (\mathbb{R}^n) \simeq \mathbb{V}_{0,\dots, 0} \oplus \mathbb{V}_{2,0\dots, 0}.
\end{equation}
It is clear that any $\SO_n(\mathbb{R})$-equivariant embedding of $\Flag(k_1,\dots,  k_p,\mathbb{R}^n)$ in $\Sym^2 (\mathbb{R}^n)$ admits a decomposition that is compatible with \eqref{eq:irrep2n}.
\end{proof}

The traceless isospectral models are the unique models with the lowest possible ambient dimension alluded to in Theorem~\ref{thm:minimal flag}.
\begin{proposition}[Minimal equivariant matrix model]\label{prop:min}
Let $0 \eqqcolon k_0  <   k_1 < \cdots <  k_p  < k_{p+1} \coloneqq n$ be integers and  $n\ge 17$. If $\mathbb{V}$ is an $\SO_n(\mathbb{R})$-module of the lowest dimension such that $\Flag(k_1,\dots,k_p,\mathbb{R}^n)$ is an $\SO_n(\mathbb{R})$-submanifold of $\mathbb{V}$,  then there is an $\SO_n(\mathbb{R})$-module isomorphism $\mathbb{V} \simeq \mathsf{S}^2_\oh(\mathbb{R}^n)$ and distinct $a_1,\dots, a_{p+1} \in \mathbb{R}$ with
\[
\sum_{i=0}^p (k_{i+1} - k_i) a_{i+1} = 0
\]
such that the following diagram commutes:
\[
\begin{tikzcd}
	{\mathbb{V}} && {\mathbb{V}_{2,0\dots, 0} \simeq \mathsf{S}^2_\oh(\mathbb{R}^n)} \\
	{\Flag(k_1,\dots,k_p,\mathbb{R}^n)} && {\Flag_{a_1,\dots,  a_{p+1}}(k_1,\dots,  k_p, n)}
	\arrow["\simeq", from=1-1, to=1-3]
	\arrow[hook, from=2-1, to=1-1]
	\arrow["\simeq", from=2-1, to=2-3]
	\arrow[hook, from=2-3, to=1-3]
\end{tikzcd}
\]
where $\Flag_{a_1,\dots,  a_{p+1}}(k_1,\dots,  k_p, n)$ is as defined in \eqref{eq:iso}.
\end{proposition}
\begin{proof}
The existence of an $\SO_n(\mathbb{R})$-module isomorphism $\varphi : \mathbb{V} \to \mathbb{V}_{2,0\dots, 0}$ is guaranteed by Theorem~\ref{thm:minimal flag}. The inclusion map $\iota : \mathbb{V}_{2,0\dots, 0} \to \mathsf{S}^2 (\mathbb{R}^n)$ is obviously an injective $\SO_n(\mathbb{R})$-module homomorphism. Given an $\SO_n(\mathbb{R})$-equivariant embedding $\varepsilon : \Flag(k_1,\dots,k_p,\mathbb{R}^n) \to \mathbb{V}$, composing these three maps gives an $\SO_n(\mathbb{R})$-equivariant embedding
\[
\iota \circ \varphi \circ \varepsilon : \Flag(k_1,\dots,k_p,\mathbb{R}^n) \to \mathsf{S}_{\oh}^2 (\mathbb{R}^n).
\]
The equivariance of $\iota \circ \varphi \circ \varepsilon$ implies that the image of $\iota \circ \varphi \circ \varepsilon$ is an $\SO_n(\mathbb{R})$-orbit of a diagonal matrix $\diag(a_1 I_{n_1},\dots,  a_{p+1} I_{n_{p+1}})$ for some distinct $a_1,\dots,  a_{p+1} \in \mathbb{R}$,  thus
\[
\iota \circ \varphi \circ \varepsilon( \Flag(k_1,\dots,k_p,\mathbb{R}^n) ) = \Flag_{a_1,\dots,  a_{p+1}}(k_1,\dots,  k_p, n).
\]
Since $\mathbb{V}_{2,0\dots, 0} \simeq \mathsf{S}^2_\oh(\mathbb{R}^n)$, any $X \in \iota \circ \varphi \circ \varepsilon( \Flag(k_1,\dots,k_p,\mathbb{R}^n) )$ is traceless and so  $a_1,\dots,  a_{p+1} $ must satisfy  $\sum_{i=0}^p (k_{i+1} - k_i) a_{i+1} = \tr(X) = 0$. 
\end{proof}
Proposition~\ref{prop:min} requires $n \ge 17$ as its proof depends on Theorem~\ref{thm:minimal flag}. Every other result in this section holds true irrespective of the value of $n$. 

As an addendum, we state an analogue of Theorem~\ref{thm:minimal flag} for the Stiefel manifold $\V(k,\mathbb{R}^n)$ of orthonormal $k$-frames in $\mathbb{R}^n$.
\begin{proposition}[Minimal equivariant embedding of Stiefel manifolds]\label{prop:minimal stiefel}
Let $k,n \in \mathbb{N}$,  $n \ge 17$, and $k < (n-1)/2$. Let $\mathbb{V}$ be an $\SO_n(\mathbb{R})$-module of the smallest dimension such that $\V(k,\mathbb{R}^n)$ is an $\SO_n(\mathbb{R})$-submanifold of $\mathbb{V}$. Then $\mathbb{V} \simeq \mathbb{V}_{1,0\dots, 0}^{\oplus k} \simeq  \mathbb{R}^{n \times k}$.
\end{proposition}
\begin{proof}
Since the usual model $\V(k,n) \coloneqq \{Y \in \mathbb{R}^{n \times k} : Y^\tp Y = I \}$ for a Stiefel manifold is an $\SO_n(\mathbb{R})$-submanifold of $\mathbb{V}_{1,0\dots, 0}^{\oplus k} \simeq \mathbb{R}^{n\times k}$, it remains to rule out $\dim \mathbb{V} < kn$ as a possibility.
We have assumed $k < (n-1)/2$, so $kn < \dim \mathbb{V}_{1,1,0\dots, 0}$. Suppose $\dim \mathbb{V} < kn$. Then by Lemma~\ref{lem:lowdim}, $\mathbb{V} = \mathbb{V}_{0,\dots, 0}^{\oplus a} \oplus \mathbb{V}_{1,0\dots, 0}^{\oplus b}$ for some nonnegative integers $a,b$ such that $a + bn < kn$.
The same argument in the proof of Theorem~\ref{thm:minimal flag} yields that $a = 0$ and $\mathbb{V} = \mathbb{V}_{1,0\dots, 0}^{\oplus b} \simeq \mathbb{R}^{n\times b}$.  For any $Y\in \mathbb{R}^{n\times b}$,
\[
\Stab(Y) = \lbrace
Q\in \SO_n(\mathbb{R}):  QY = Y
\rbrace \simeq \SO_{n-q} (\mathbb{R}),
\]
where $q \coloneqq \rank(Y) \le b < k$.  Hence $\Stab(Y) \ne \SO_{n-k}(\mathbb{R})$ and it is not possible for $\V(k,\mathbb{R}^n) \simeq \SO_n(\mathbb{R})/\SO_{n-k}(\mathbb{R})$ to be an $\SO_{n}(\mathbb{R})$-submanifold of $\mathbb{V}$.
\end{proof}

\section{Conclusion}

The classical embeddings of Whitney, Nash, and Mostow--Palais embed a manifold in $\mathbb{R}^n$. The key to our effective bounds in this article is that we embed our manifold in $\mathbb{R}^{m \times n}$. While they are no different as vector spaces, $\mathbb{R}^{m \times n} \simeq \mathbb{R}^{mn}$, matrices are endowed with far richer structures --- we may multiply or decompose them; impose orthogonality or symmetry on them; calculate their determinant, norm, or rank; find their eigen- or singular values and vectors; among a myriad of yet other features.

This is not unlike group representation theory where we embed an abstract group into a matrix group; the parallel here is that we embed an abstract manifold into a manifold of matrices. While group representation theory permeates every area of pure mathematics, such ``manifold representation theory'' has, as far as we know, only been popular in computational mathematics. We hope that our work here shows that it can be useful for investigations in pure mathematics too.

\bibliographystyle{abbrv}

\end{document}